\documentclass[11pt]{article}
\usepackage{bbm}
\usepackage{mathrsfs}
\usepackage{amsfonts}
\usepackage{amssymb}
\usepackage{amsmath,theorem,float}
\usepackage{color}
\usepackage{multirow}

\newtheorem{theorem}{Theorem}[section]
\newtheorem{prop}[theorem]{Proposition}

\newtheorem{definition}[theorem]{Definition}
\newtheorem{remark}[theorem]{Remark}
\newtheorem{example}[theorem]{Example}

\def\N{{\mathbb{N}}}

\def\Y{{\mathbb{Y}}}

\def\rank{\hbox{\rm{rank}}}

\def\D{{\sigma}}

\def\ord{\hbox{\rm{ord}}}
\def\fp{{\mathfrak{p}}}
\def\trdeg{\hbox{\rm{trdeg}}}
\def\Frac{\hbox{\rm{Frac}}}

\begin{document}
\title{Difference Index of Quasi-regular Difference Algebraic Systems}
\author{Jie Wang}
\date{\today}
\maketitle

\begin{abstract}
This paper is devoted to studying difference indices of quasi-regular difference algebraic systems. We give the definition of difference indices through a family of pseudo-Jacobian matrices. Some properties of difference indices are proved. In particular, a Jacobi-type upper bound for the sum of the order and the difference index is given. As applications, an upper bound of the Hilbert-Levin regularity and an upper bound of the order for difference ideal membership problem are deduced.
\end{abstract}

\section{Introduction}
There are several definitions of differential indices of a differential algebraic system in the literature (see for instance \cite{d-ca}, \cite{d-ku}, \cite{d-le}, \cite{d-pa}, \cite{d-se}). Particularly in \cite{iqr} and \cite{ldi}, D'Alfonso, Jeronimo, Massaccesi and Solern\'o introduced the notion of $\mathfrak{P}$-differential indices for quasi-regular differential algebraic systems through a family of pseudo-Jacobian matrices. Although they are not completely equivalent, in each case they represent a measure of the implicitness of the given system. It seems that the corresponding difference index of a difference algebraic system has not been studied yet. In this paper, we first give the definition of difference indices for quasi-regular difference algebraic systems, following the method used in \cite{iqr} and \cite{ldi}.

Suppose $F=\{f_1,\ldots,f_r\}$ is a couple of difference polynomials, $\Delta$ is the difference ideal generated by $F$, and $\fp$ is a minimal reflexive prime difference ideal over $\Delta$. For an element $a$ in a difference field, denote $a^{[k]}=\{a,\D(a),\ldots,\D^k(a)\}$. Then we say the system $F$ is {\em quasi-regular} at $\fp$ if for every positive integer $i$, the Jacobian matrix of the polynomials $f_1^{[k-1]},\ldots,f_r^{[k-1]}$ with respect to the set of variables $\Y^{[k-1+e]}$ has full row rank and $\Delta$ is reflexive. Through a family of pseudo-Jacobian matrices, we can give the definition of the difference index of a quasi-regular difference algebraic system, which is called the {\em $\fp$-difference index}. As usual, its definition follows from a certain chain which eventually becomes stationary. Similarly to the case of $\mathfrak{P}$-differential indices in \cite{iqr}, the chain is established by the sequence of ranks of certain Jacobian submatrices associated with the system $F$. Assume $\omega$ is the $\fp$-difference index of the system $F$. It turns out that for every $i\ge e-1$ ($e$ is the highest order of $F$), $\omega$ satisfies:
\[\Delta_{i-e+1+\omega}\cap A_i=\Delta\cap A_i,\]
where $A_i$ is the polynomial ring in the variables with orders no more than $i$, which meets our expectation for difference indices.

This approach enables us to give an upper bound for the sum of the order and the $\fp$-difference index of a quasi-regular system. Based on this, we can give several applications of $\fp$-difference indices, including an upper bound of the Hilbert-Levin regularity and an upper bound of orders for difference ideal membership problem.

The paper will be organized as follows. In Section 2, we list some basic notions from difference algebra which will be used later. In Section 3, we give the definition of quasi-regular difference algebraic systems. In Section 4, we introduce a family of pseudo-Jacobian matrices and give the definition of $\fp$-difference indices through studying the ranks of them. In Section 5, a Jacobi-type upper bound for the sum of the order and the $\fp$-difference index is given. In Section 6, several applications of $\fp$-difference indices are given. In Section 7, we give an example.

\section{Preliminaries}
A {\em difference ring} or {\em $\sigma$-ring} for short $(R,\sigma)$, is a commutative ring $R$ together with a ring endomorphism $\sigma\colon R\rightarrow R$. If $R$ is a field, then we call it a {\em difference field}, or a {\em $\sigma$-field} for short. We usually omit $\sigma$ from the notation, simply refer to $R$ as a $\sigma$-ring or a $\sigma$-field. In this paper, $K$ is always assumed to be a $\D$-field of characteristic $0$.

\begin{definition}
Let $R$ be a $\D$-ring. An ideal $I$ of $R$ is called a {\em $\D$-ideal} if for $a\in R$, $a\in I$ implies $\D(a)\in I$. Suppose $I$ is a $\D$-ideal of $R$, then $I$ is called
\begin{itemize}
  \item {\em reflexive} if $\D(a)\in I$ implies $a\in I$ for $a\in R$;
  \item {\em $\D$-prime} if $I$ is reflexive and a prime ideal as an algebraic ideal.
\end{itemize}
\end{definition}

For a subset $F$ in a $\D$-ring, we denote $[F]$ the $\D$-ideal generated by $F$. Let $K$ be a $\sigma$-field. Suppose $\Y=\{y_1,\ldots,y_n\}$ is a set of $\sigma$-indeterminates over $K$. Then the {\em $\sigma$-polynomial ring} over $K$ in $\Y$ is the polynomial ring in the variables $\Y,\sigma(\Y),\sigma^2(\Y),\ldots$. It is denoted by
\[K\{\Y\}=K\{y_1,\ldots,y_n\}\]
and has a natural $K$-$\sigma$-algebra structure.

For more details about difference algebra, please refer to \cite{wibmer}. 

\section{Quasi-regular difference algebraic systems}
Let $K$ be a $\D$-field. Let $a$ be an element in a $\D$-extension field of $K$, $S$ a set of elements in a $\D$-extension field of $K$, and $i\in\N$. Denote $a^{(i)}=\D^i(a), a^{[i]}=\{a,a^{(1)},\ldots,a^{(i)}\}$, $S^{(i)} = \cup_{a\in S} \{a^{(i)}\}$ and $S^{[i]} = \cup_{a\in S} a^{[i]}$. For the $\sigma$-indeterminates $\Y=\{y_1,\ldots,y_n\}$ and $i\in\N$, we will treat the elements of $\Y^{[i]}$ as algebraic indeterminates, and $K[\Y^{[i]}]$ is the polynomial ring in $\Y^{[i]}$.

Throughout the paper let $F=\{f_1,\ldots,f_r\}\subset K\{\Y\}$ be a system of difference polynomials over $K$ and $\fp\subseteq K\{\Y\}$ a $\D$-prime ideal minimal over $[F]$. Let $\epsilon_{ij}:=\ord_{y_j}(f_i)$ which is the order of $f_i$ with respect to $y_j$ and denote $e:=\max\{\epsilon_{ij}\}$ for the maximal difference order which occurs in $F$. We assume that $F$ actually involves difference operator, i.e. $e\ge1$.
We introduce also the following auxiliary polynomial rings and ideals: for every $k\in\N$, we denote $A_k$ the polynomial ring $A_k:=K[\Y^{[k]}]$ and $\Delta_k:=(f_1^{[k-1]},\ldots,f_r^{[k-1]})\subseteq A_{k-1+e}$. We set $\Delta_0:=(0)$ by definition.

For each non-negative integer $k$, we write $B_k$ for the local ring obtained from $A_k$ after localization at the prime ideal $A_k\cap\fp$ and we denote $\fp_k:=A_{k-1+e}\cap\fp$. Since each $A_k$ is a polynomial ring, the localizations $B_k$ are regular rings. For the sake of simplicity, we preserve the notation $\Delta_k$ for the ideal generated by $f_1^{[k-1]},\ldots,f_r^{[k-1]}$ in the local ring $B_{k-1+e}$ and denote $\Delta$ the $\D$-ideal generated by $F$ in $K\{\Y\}_{\fp}$.

\begin{definition}
We say the system $F$ is {\em quasi-regular} at $\fp$ if for every positive integer $k$, the Jacobian matrix of the polynomials $f_1^{[k-1]},\ldots,f_r^{[k-1]}$ with respect to the set of variables $\Y^{[k-1+e]}$ has full row rank over the domain $A_{k-1+e}/\fp_k$ and $\Delta$ is reflexive.
\end{definition}

This condition can be easily rephrased in terms of K\"ahler differentials saying that the differentials $\{\mathrm{d} f_i^{(k)}, 1\le i\le r, k\in\N\}\subset\Omega_{K\{\Y\}/K}$ are a $K\{\Y\}/\fp$-linearly independent set in $\Omega_{K\{\Y\}/K}\otimes_{K\{\Y\}} K\{\Y\}/\fp$.

\begin{remark}\label{pd-re}
If the $\D$-ideal $[F]\subseteq K\{\Y\}$ is already a $\D$-prime ideal, the minimality of $\fp$ implies that $\fp=[F]$ and all our results remain true considering the rings $A_k$ and the $\D$-ideal $[F]$ without localization. In this case if $F$ is quasi-regular at $[F]$ we will say simply that $F$ is quasi-regular.
\end{remark}

In this paper, we always assume that $F$ is a difference algebraic system which is quasi-regular at $\fp$.

\begin{prop}\label{pd-prop}
Let $F$ be a difference algebraic system which is quasi-regular at $\fp$. For $k\in\N^*$, we have:
\begin{enumerate}
  \item $f_1^{[k-1]},\ldots,f_r^{[k-1]}$ is a regular sequence in the local ring $B_{k-1+e}$ and generates a prime ideal.
  \item In the localized ring $K\{\Y\}_{\fp}$, $\Delta$ agrees with $\fp K\{\Y\}_{\fp}$.
  \item If $\kappa$ denotes the residue class field of $\fp$, the difference transcendence degree of $\kappa$ over $K$ is $n-r$.
\end{enumerate}
\end{prop}
\begin{proof}
The proof is similar to \cite[Proposition 3]{iqr}.
\end{proof}


\section{The definition of $\fp$-difference index}
Following \cite{iqr}, we introduce a family of pseudo-Jacobian matrices which we need in order to define the concept of difference index. For a matrix $E$ over $K$, we use $E^{(i)}$ to denote the matrix whose elements are the $i$-th transform of the corresponding elements of $E$.
\begin{definition}
For each $k\in\N$ and $i\in\N_{\ge e-1}$ (i.e. $i\in\N$ and $i\ge e-1$), we define the $kr\times kn$-matrix $J_{k,i}$ as follows:
\begin{align*}
J_{k,i}:&=\frac{\partial{(F^{(i-e+1)},F^{(i-e+2)},\ldots,F^{(i-e+k)}})}{\partial{(\Y^{(i+1)},\Y^{(i+2)},\ldots,\Y^{(i+k)})}}\\
&=\begin{pmatrix}
\frac{\partial F^{(i-e+1)}}{\partial \Y^{(i+1)}}&0&\cdots&0\\ \frac{\partial F^{(i-e+2)}}{\partial \Y^{(i+1)}}&\frac{\partial F^{(i-e+2)}}{\partial \Y^{(i+2)}}&\cdots&0\\ \vdots&\vdots&\ddots&\vdots\\ \frac{\partial F^{(i-e+k)}}{\partial \Y^{(i+1)}}&\frac{\partial F^{(i-e+k)}}{\partial \Y^{(i+2)}}&\cdots&\frac{\partial F^{(i-e+k)}}{\partial \Y^{(i+k)}}
\end{pmatrix},
\end{align*}
where each $\frac{\partial F^{(p)}}{\partial \Y^{(q)}}$ denotes the Jacobian matrix $(\partial(f_1^{(p)},\ldots,f_r^{(p)})/\partial(y_1^{(q)},\ldots,y_n^{(q)}))_{r\times n}$.
\end{definition}

Since the partial derivative operator and the difference operator are commutative, we have
\[J_{k,i}=\begin{pmatrix}
(\frac{\partial F}{\partial \Y^{(e)}})^{(i-e+1)}&0&\cdots&0\\ (\frac{\partial F}{\partial \Y^{(e-1)}})^{(i-e+2)}&(\frac{\partial F}{\partial \Y^{(e)}})^{(i-e+2)}&\cdots&0\\ \vdots&\vdots&\ddots&\vdots\\ (\frac{\partial F}{\partial \Y^{(e-k+1)}})^{(i-e+k)}&(\frac{\partial F}{\partial \Y^{(e-k+2)}})^{(i-e+k)}&\cdots&(\frac{\partial F}{\partial \Y^{(e)}})^{(i-e+k)}
\end{pmatrix}.\]

Note that $J_{k,i+1}=J_{k,i}^{(1)}$ and
\begin{align}
J_{k+1,i}&=\begin{pmatrix}
\text{\Large $J_{k,i}$}&\begin{matrix}0\\ \vdots\\ 0\end{matrix}\\ \begin{matrix}(\frac{\partial F}{\partial \Y^{(e-k)}})^{(i-e+k+1)}&\cdots&(\frac{\partial F}{\partial \Y^{(e-1)}})^{(i-e+k+1)}\end{matrix}&(\frac{\partial F}{\partial \Y^{(e)}})^{(i-e+k+1)}
\end{pmatrix}\label{di-equ2} \\
&=\begin{pmatrix}
(\frac{\partial F}{\partial \Y^{(e)}})^{(i-e+1)}&\begin{matrix}0&\cdots&0\end{matrix}\\
\begin{matrix}(\frac{\partial F}{\partial \Y^{(e-1)}})^{(i-e+2)}\\ \vdots\\ (\frac{\partial F}{\partial \Y^{(e-k)}})^{(i-e+k+1)}\end{matrix}&\text{\Large $J_{k,i}^{(1)}$}
\end{pmatrix}.\label{di-equ3}
\end{align}

\begin{definition}
For $k\in\N$ and $i\in\N_{\ge e-1}$, we define $\mu_{k,i}\in\N$ as follows:
\begin{itemize}
  \item $\mu_{0,i}:=0$;
  \item $\mu_{k,i}:=\dim_{\kappa}\ker(J_{k,i}^{\tau})$, for $k\ge1$, where $J_{k,i}^{\tau}$ denotes the usual transpose of the matrix $J_{k,i}$ and $\kappa$ denotes the residue class field of $\fp$. In particular $\mu_{k,i}=kr-\rank_{\kappa}(J_{k,i})$.
\end{itemize}
\end{definition}

\begin{prop}
Let $k\in\N$ and $i\in\N_{\ge e-1}$. Then $\mu_{k,i}=\mu_{k,i+1}$.
\end{prop}
\begin{proof}
Since $J_{k,i+1}=J_{k,i}^{(1)}$ for any $k\in\N$ and any $i\in\N_{\ge e-1}$, we just need to show that $J_{k,i}^{(1)}$ and $J_{k,i}$ have the same rank. This is obvious since maximal nonzero minors of $J_{k,i}^{(1)}$ and $J_{k,i}$ have the same order.
\end{proof}

The previous proposition shows that the sequence $\mu_{k,i}$ does not depend on the index $i$. Therefore, in the sequel, we will write $\mu_k$ instead of $\mu_{k,i}$, for any $i\in\N_{\ge e-1}$.

Denote $\kappa(\Delta_k)$ the residue class field of $\Delta_k$ in the ring $B_{k-1+e}$, $\kappa(\fp_k)$ the residue class field of $\fp_k$ in the ring $A_{k-1+e}$ and $\kappa$ the residue class field of $\fp$. As an additional hypothesis on the system $F$, we assume that the rank of the matrix $J_{k,i}$ over $\kappa(\Delta_{i-e+1+k+s})$ does not depend on $s$, where $s\in\N$. That is to say, we assume the rank of the matrix $J_{k,i}$ considered alternatively over $\kappa(\Delta_{i-e+1+k})$, or over $\kappa(\fp_{i-e+1+k})$, or over $\kappa$ is always the same.
\begin{prop}\label{di-prop1}
Let $k\in\N$ and $i\in\N_{\ge e-1}$. Then:
\begin{enumerate}
  \item The transcendence degree of the field extension
  $$\Frac(B_{i}/(\Delta_{i-e+1+k}\cap B_{i}))\hookrightarrow \Frac(B_{i+k}/\Delta_{i-e+1+k})$$
  is $k(n-r)+\mu_{k}$.
  \item The following identity holds:
  $$\trdeg_K(\Frac(B_{i}/(\Delta_{i-e+1+k}\cap B_{i})))=(n-r)(i+1)+er-\mu_{k}.$$
\end{enumerate}
\end{prop}
\begin{proof}
The proof is similar to \cite[Proposition 6]{iqr}.
\end{proof}

\begin{prop}\label{di-prop2}
The sequence $(\mu_{k})_{k\in\N}$ is non-decreasing and verifies the inequality
\begin{equation}\label{di-equ1}
\sum_{j=1}^r\min\{k,e-e_j\}\le\mu_{k}\le\min\{k,e\}r.
\end{equation}
In particular, there exists $k\in\N$, $0\le k\le e+\sum_{j=1}^r e_j$, such that $\mu_{k}=\mu_{k+1}$.
\end{prop}
\begin{proof}
Fix an index $i\in\N_{\ge e-1}$. From (\ref{di-equ2}), it is easy to see that $\ker(J_{k,i}^{\tau})\times\{0\}^r\subseteq\ker(J_{k+1,i}^{\tau})$ for every $k\in\N$. Then the fact $(\mu_{k})_{k\in\N}$ is a non-decreasing sequence follows immediately.

For every $k\in\N$, the matrix $J_{k,i}$ has $kr$ rows. Therefore, $\dim_{\kappa}\ker(J_{k,i}^{\tau})\le kr$. On the other hand, due to Proposition \ref{di-prop1}, we have that $\trdeg_K(\Frac(B_i/\Delta_{i-e+1+k}\cap B_i))=(n-r)(i+1)+er-\mu_{k}$. Now, $\trdeg_K(\Frac(B_i/\Delta_{i-e+1+k}\cap B_i))\ge \trdeg_K(\Frac(B_i/\Delta\cap B_i))$, since $\Delta_{i-e+1+k}\cap B_i \subseteq\Delta\cap B_i$, and so, the fact that the difference dimension of $\Delta$ is $n-r$ implies that $\trdeg_K(\Frac(B_i /\Delta\cap B_i))\ge(n-r)(i+1)$. Hence, $\mu_{k}\le er$ holds.

In order to show the other inequality, we observe that, since the order of the polynomial $f_j$ is $e_j(1\le j\le r)$, the partial derivatives $\frac{\partial{f_j}}{\partial{\Y^{(q)}}}$ are zeros for $q>e_j$. So, each polynomial $f_j$ induces $k$ null rows if $e-k+1>e_j$ or $e-e_j$ null rows if $e-k+1\le e_j$ in the matrix $J_{k,i}$. Equivalently, $f_j$ induces $\min\{k,e-e_j\}$ many null rows in the matrix $J_{k,i}$. We conclude that the matrix $J_{k,i}$ has at least $\sum^r_{j=1}\min\{k,e-e_j\}$ null rows. Thus, the dimension of the kernel of the transpose matrix $J^{\tau}_{k,i}$ (i.e. $\mu_{k}$) is at least $\sum^r_{j=1}\min\{k,e-e_j\}$. The second assertion follows directly from the fact that for every $k\ge e$, the inequality (\ref{di-equ1}) reads $\sum^r_{j=1}(e-e_j)\le\mu_{k}\le er$.
\end{proof}

\begin{theorem}\label{di-thm1}
Fix an index $i\in\N_{\ge e-1}$ and let $k_0\in\N$ be the minimum of $k$ such that $\mu_{k+1}=\mu_{k}$ (this minimum is well defined due to Proposition \ref{di-prop2}). Then $\mu_{k}=\mu_{k_0}$ for every $k\ge k_0$.
\end{theorem}
\begin{proof}
The result is clear for $k_0=0$: in this case, $\mu_{1}=0$, which is equivalent to the fact that the matrix $J_{1,i}$ has full row rank. We conclude that $J_{k,i}$ has full row rank too or, equivalently, that $\mu_{k}=0$ for all $k$.

Now, let us assume that $k_0\ge1$. It suffices to show that the equality $\mu_{k}=\mu_{k-1}$ for an arbitrary index $k\ge2$, implies $\mu_{k+1}=\mu_{k}$. In the sequel, for a vector $v\in\kappa^{lr}$ we will write its description as a block vector $v=(v_1,\ldots,v_l)$ with $v_j\in\kappa^r$. Due to the recursive relation (\ref{di-equ2}), the identity $\ker(J^t_{k,i})\times\{0\}^r=\ker(J^t_{k+1,i})\cap\{v_{k+1}=\mathbf{0}\}$ holds in $\kappa^{(k+1)r}$ for every $k\in\N$ and so, the equality $\mu_{k}=\mu_{k+1}$ is equivalent to the inclusion $\ker(J^t_{k+1,i})\subseteq\{v_{k+1}=\mathbf{0}\}$. Then, the theorem is a consequence of the following recursive principle:

Claim: For all $k\in\N$, $\ker(J^{\tau}_{k,i})\subseteq\{v_k=\mathbf{0}\}$ implies $\ker(J^{\tau}_{k+1,i})\subseteq\{v_{k+1}=\mathbf{0}\}$.\\
{\em Proof of the claim.} Suppose $w=(w_1,\ldots,w_{k+1})^{\tau}$ is a solution of $J_{k+1,i}^{\tau}$, then by (\ref{di-equ3}), we have $(w_2,\ldots,w_{k+1})\cdot J_{k,i}^{(1)}=0$. Since $\ker(J^{\tau}_{k,i})\subseteq\{v_k=\mathbf{0}\}$, $J^{\tau}_{k,i}$ can be transformed to a upper triangular matrix with the last $k$ rows is an identity matrix through the elementary row transformations, and so can $(J_{k,i}^{(1)})^{\tau}$. It follows $w_{k+1}=\mathbf{0}$ which proves the claim.
\end{proof}

\begin{definition}
By Theorem \ref{di-thm1}, there exists $\omega\in\N$ such that $\mu_k<\mu_{k+1}$ for all $k<\omega$ and $\mu_k=\mu_{k+1}$ for all $k\ge\omega$. Such $\omega$ is called the {\em $\fp$-difference index} of the system $F$. If $[F]$ is itself a $\D$-prime ideal, we say simply the difference index of $F$.
\end{definition}

It is obvious from the construction that $\omega$ is depending on the choice of the minimal $\D$-prime ideal $\fp$ over $[F]$. However, we will prove some properties of $\omega$ which meet our expectation for difference indices.

\section{Properties of $\fp$-difference index}
\subsection{Manifold of constraints}
A remarkable property associated with most differentiation indices is that they provide an upper bound for the number of derivatives of the system needed to obtain all the constraints that must be satisfied by the solutions of the system. This case is also suitable for the $\fp$-difference indices defined above.
\begin{theorem}\label{mc-tm}
Let $\omega\in\N$ be the $\fp$-difference index of the system $F$. Then, for every $i\in\N_{\ge e-1}$, the equality of ideals
\[\Delta_{i-e+1+\omega}\cap B_i=\Delta\cap B_i\]
holds in the ring $B_i$. Furthermore, for every $i\in\N_{\ge e-1}$, the $\fp$-difference index $\omega$ verifies: $\omega=\min\{h \in\N:\Delta_{i-e+1+h}\cap B_i=\Delta\cap B_i\}$.
\end{theorem}
\begin{proof}
Fix an index $i\in\N_{\ge e-1}$. Let us consider the increasing chain $(\Delta_{i-e+1+k}\cap B_i)_{k\in\N}$ of prime ideals in the ring $B_i$. From Proposition \ref{di-prop1}, for every $k\in\N$, we have
\begin{equation}\label{pdi-eq}
\trdeg_K(\Frac(B_{i}/(\Delta_{i-e+1+k}\cap B_{i})))=(n-r)(i+1)+er-\mu_{k}.
\end{equation}
Since $\mu_k$ is stationary for $k\ge\omega$ (Theorem \ref{di-thm1}), all the prime ideals $\Delta_{i-e+1+k}\cap B_{i}$ have the same dimension for $k\ge\omega$ and the chain of prime ideals becomes stationary for $k\ge\omega$.

It only remains to prove that the largest ideal of the chain coincides with $\Delta\cap B_i$. One inclusion is obvious. For the other, let $f$ be an arbitrary element of $\Delta\cap B_i$, then there exist difference polynomials $h,a_{lj}\in K\{\Y\}, h\notin\fp$ such that
\[f=\sum^r_{l=1}\sum_{j}\frac{a_{lj}f^{(j)}_l}{h}.\]
Let $N$ be the maximal order of the variables $\Y$ appearing in this equality. Then we have $f\in \Delta_{N-e+1}\subseteq B_N$ and hence $f\in\Delta_{N-e+1}\cap B_i$. Since the above chain of ideals is stationary for $k\ge\omega$, $f\in\Delta_{i-e+1+\omega}\cap B_i$. This completes the proof of the first assertion of the Theorem.

In order to prove the second part of the statement, for each $i\in\N_{\ge e-1}$, let $h_i$ be the smallest non-negative integer such that $\Delta_{i-e+1+h_i}\cap B_i=\Delta\cap B_i$. By the choice of $h_i$, the transcendence degrees $\trdeg_K(\Frac(B_{i}/(\Delta_{i-e+1+k}\cap B_{i})))$ coincide for $k\ge h_i$, and so by (\ref{pdi-eq}), $\mu_k$ is constant for $k\ge h_i$. This implies that $\omega\le h_i$. The equality follows from the first part of the statement and the minimality of $h_i$.
\end{proof}
\begin{remark}
Taking $i=e-1$ in the last assertion of Theorem \ref{mc-tm}, we have the following alternative definition of the $\fp$-difference index:
\[\omega=\min\{h\in\N:\Delta_h\cap B_{e-1}=\Delta\cap B_{e-1}\}.\]
\end{remark}

\subsection{The order of $\fp$}
The following proposition reveals a connection between $\mu_k$ and the order of $\fp$.
\begin{prop}\label{ord}
Assume that $F$ is a difference algebraic system which is quasi-regular at $\fp$ and $\omega$ is the $\fp$-difference index of $F$. Then $\ord(\fp)=er-\mu_{\omega}.$
\end{prop}
\begin{proof}
Fix an index $i\in\N_{\ge e-1}$. By Theorem \ref{mc-tm}, for $k\ge\omega$, $\Delta_{i-e+1+k}\cap B_i=\Delta\cap B_i$. Therefore, for $k\ge\omega$, by Proposition \ref{di-prop1} and Theorem \ref{di-thm1},
\begin{align*}
\trdeg_K(\Frac(B_{i}/(\Delta\cap B_{i})))&=\trdeg_K(\Frac(B_{i}/(\Delta_{i-e+1+k}\cap B_{i})))\\
&=(n-r)(i+1)+er-\mu_{k}.
\end{align*}
On the other hand, since $\Frac(B_{i}/(\Delta\cap B_{i}))=\Frac(A_{i}/(\fp\cap A_{i}))$, by the dimension polynomial of $\fp$ (see for instance \cite[Chapter 5]{wibmer}),
\begin{align*}
\trdeg_K(\Frac(B_{i}/(\Delta\cap B_{i})))&=\D\textrm{-}\dim(\fp)(i+1)+\ord(\fp)\\
&=(n-r)(i+1)+\ord(\fp),
\end{align*}
where $\D\textrm{-}\dim(\fp)=n-r$ by Proposition \ref{pd-prop}. It follows $\ord(\fp)=er-\mu_{\omega}.$
\end{proof}

\subsection{Jacobi-type bounds}
Jacobi introduced a parameter associated with the orders of derivations in a differential algebraic system and conjectured an upper bound for the order of the system in terms of this number. Cohn generalized this to a difference algebraic system.

For a difference algebraic system $F$, we introduce an auxiliary integer matrix $\mathcal{E}_0:=(\epsilon_{ij})_{r\times n}$ whose entries are the orders $\epsilon_{ij}$ of $f_i$ with respect to the variable $y_j$ appearing in $f_i$ and $0$ if the variable $y_j$ does not appear in $f_i$.
\begin{definition}
Let $A\in \N^{r\times n}$, $r\le n$, be an integer matrix. The Jacobi number of $A$ is defined to be
$$J(A):=\max\{\sum_{i=1}^ra_{i\tau(i)}\mid \tau\colon\{1,\ldots,r\}\rightarrow\{1,\ldots,n\}\textrm{ is an injection }\}.$$
\end{definition}

We have the following the Jacobi-type bound for the sum of the $\fp$-difference index and the order of $\fp$. Since the proof is very similar to \cite[Theorem 15]{iqr}, we will omit it.
\begin{theorem}\label{pdi-thm}
Let $F$ be a quasi-regular system at $\fp$. Then, the $\fp$-difference index $\omega$ of the system $F$ and the order $\ord(\fp)$ of $\fp$ satisfy
$$\omega+\ord(\fp)\le J(\mathcal{E}_0)+e-\min\{\epsilon_{ij}\}.$$
\end{theorem}

\section{Applications of $\fp$-difference index}
\subsection{The Hilbert-Levin regularity}
For a $\D$-prime ideal $\fp$, the polynomial $\varphi(i)=\D\textrm{-}\dim(\fp)(i+1)+\ord(\fp)$ is known as the dimension polynomial of $\fp$ (see for instance \cite[Chapter 5]{wibmer}). The minimum of the indices $i_0$ such that $\varphi(i)=\trdeg_K(\Frac(A_i/(A_i\cap\fp)))$ for all $i\ge i_0$ is called the {\em Hilbert-Levin regularity} of $\fp$. The results developed on $\fp$-difference indices enable us to give an upper bound of the Hilbert-Levin regularity of $\fp$.
\begin{theorem}
The Hilbert-Levin regularity of the $\D$-prime ideal $\fp$ is bounded by $e-1$.
\end{theorem}
\begin{proof}
Since for all $i\in\N$, we have $\Frac(A_i/(A_i\cap\fp))=\Frac(B_i/(B_i\cap\Delta))$. Therefore, $\trdeg_K(\Frac(A_i/(A_i\cap\fp)))=\trdeg_K (\Frac(B_i/(B_i\cap\Delta)))$ and so, it is enough to show that, for all $i\ge e-1$, $\trdeg_K(\Frac(B_i/(B_i\cap\Delta)))+\D\textrm{-}\dim(\fp)=\trdeg_K(\Frac(B_{i+1}/(B_{i+1}\cap\Delta)))$.

Fix an index $i\ge e-1$. By Theorem \ref{mc-tm}, we have that $\Delta\cap B_i=\Delta_{i-e+1+\omega}\cap B_i$
and $\Delta\cap B_{i+1}=\Delta_{i-e+2+\omega}\cap B_{i+1}$. Thus, by Proposition \ref{di-prop1}, we obtain:
\begin{align*}
\trdeg_K(\Frac(B_{i+1}/(\Delta\cap B_{i+1})))&=(n-r)(i+2)+er-\mu_{\omega},\\
\trdeg_K(\Frac(B_{i}/(\Delta\cap B_{i})))&=(n-r)(i+1)+er-\mu_{\omega},
\end{align*}
Hence, the result holds.
\end{proof}

\subsection{The ideal membership problem}
It is well-known that in polynomial algebra, the ideal membership problem is to decide if a given element $f\in A$ belongs to a fixed ideal $I\subseteq A$ for an arbitrary commutative ring $A$, and, in the affirmative case, representing $f$ as a linear combination with polynomial coefficients of a given set of generators of $I$.

The ideal membership problem also exists in differential algebra and difference algebra. But unlike the case in polynomial algebra, this problem is undecidable for arbitrary ideals in differential algebra (see \cite{ga-crd}) and difference algebra. However, there are special classes of differential ideals for which the problem is decidable, in particular the class of radical differential ideals (\cite{etda}, see also \cite{rdi}).

When it comes to the representation problem, the differential case or the difference case involves another additional ingredient: the order $N$ of derivation or transform of the given generators of $I$ needed to write an element $f\in I$ as a polynomial linear combination of the generators and their first $N$ total derivatives or total transforms. The known order bounds seem to be too big, even for radical ideals (see for instance \cite{gp-bod}, where an upper bound in terms of the Ackerman function is given, or \cite{gu}, a better and more explicit upper bound). In \cite{iqr}, an order bound for quasi-regular differential algebraic systems is given, due to the properties of differential indices defined in the same paper. However, it seems that there does not exist any results on the corresponding bound in the difference case. By virtue of Theorem \ref{mc-tm}, we are able to give an order bound for the membership problem of a quasi-regular difference system.

The following ideal membership theorem for polynomial rings will be used.
\begin{theorem}(\cite{di}, Theorem 5.1)\label{adi-thm}
Let $K$ be a field and $g_1,\ldots,g_s\in K[y_1,\ldots,y_n]$ be a complete intersection of polynomials whose total degrees are bounded by an integer $d$. Let $g\in K[y_1,\ldots,y_n]$ be another polynomial. Then the following conditions are equivalent:\\
1. $g$ belongs to the ideal generated by $g_1,\ldots,g_s$;\\
2. there exist polynomials $a_1,\ldots,a_s$ such that $g=\sum_{j=1}^sa_jg_j$ and $\deg(a_jg_j)\le d^s+\deg(g)$ for $1\le j\le s$.
\end{theorem}

We have the following effective ideal membership theorem for quasi-regular difference algebraic systems:
\begin{theorem}
Suppose $F$ be a quasi-regular difference algebraic system in the sense of Remark \ref{pd-re}. Let $D$ be an upper bound for the total degrees of $f_1,\ldots,f_r$. Let $f\in K\{\Y\}$ be an arbitrary difference polynomial in the difference ideal $[F]$. Set $N:=\omega+\max\{-1,\ord(f)-e\}$, where $\omega$ is the difference index of $F$. Then, a representation
\[f=\sum_{1\le i\le r,0\le j\le N}g_{ij}f^{(j)}_i\]
holds in the ring $A_{N+e}$, where each polynomial $g_{ij}f_i^{(j)}$ has total degree bounded by $\deg(f)+D^{r(N+1)}$.
\end{theorem}
\begin{proof}
The upper bound on the order of transforms of the polynomials $f_1,\ldots,f_r$ is a direct consequence of Theorem \ref{mc-tm} applied to $i:=\max\{e-1, \ord(f)\}$. The degree upper bound for the polynomials $g_{ij}f_i^{(j)}$ follows from Proposition \ref{pd-prop} and Theorem \ref{adi-thm}.
\end{proof}
\begin{remark}
From Theorems \ref{mc-tm} and \ref{pdi-thm}, we have for every $i\in\N_{\ge e-1}$ , the equality
$\Delta_{i+1+J(\mathcal{E}_0)-\min\{\epsilon_{ij}\}}\cap A_i=\Delta\cap A_i$ holds. So it suffices to take $N:=J(\mathcal{E}_0)-\min\{\epsilon_{ij}\}+\max\{\ord(f),e-1\}$ to get more explicit upper bounds of the order and the degree in the above ideal membership problem.
\end{remark}

\section{An example}
\begin{example}
Notations follow as before. Consider the difference algebraic system $F=\{y_1^{(1)}-y_1y_3,y_2^{(1)}-y_2y_3,y_1+y_2-1\}\subseteq A=K\{y_1,y_2,y_3\}$. Then $\Delta=[F]$ is a $\D$-prime ideal and $F$ is a quasi-regular system in the sense of Remark \ref{pd-re}. The corresponding matrices $J_{k,0},k=1,2,3,\ldots$ are
\[\begin{pmatrix}
1&0&0&&&&&&&&&&&&\\0&1&0&&&&&&&&&&&&\\0&0&0&&&&&&&&&&&\\-1&0&-y_1&1&0&0&&&&&&&&\\0&-1&-y_2&0&1&0&&&&&&&&&\\1&1&0&0&0&0&&&&&&&&&
\\&&&-1&0&-y_1&1&0&0&&&&&&\\&&&0&-1&-y_2&0&1&0&&&&&&\\&&&1&1&0&0&0&0&&&&&&\\&&&&&&-1&0&-y_1&1&0&0&&&\\&&&&&&0&-1&-y_2&0&1&0&&&
\\&&&&&&1&1&0&0&0&0&&&\\&&&&&&&&&&&&&&\\&&&&&&&&&&\cdots&&&\cdots&\\&&&&&&&&&&&&&&
\end{pmatrix}.\]
Since $y_1^{(i)}=y_1,y_2^{(i)}=y_2,y_3^{(i)}=1$ in the ring $A/\Delta$ for all $i\in\N$, we have replaced $y_1^{(i)},y_2^{(i)},y_3^{(i)}$ by $y_1,y_2,1$ respectively in $J_{k,0}$ for all $i\in\N$. It can be computed that $\rank(J_{1,0})=2$, $\rank(J_{2,0})=4$, $\rank(J_{3,0})=7$, so $\mu_1=1,\mu_2=2,\mu_3=2$ and hence the difference index of the system $F$ is $\omega=2$. One can check that $\Delta_2\cap A_0=\Delta\cap A_0$.

The matrix $\mathcal{E}_0=\begin{pmatrix}1&0&0\\0&1&0\\0&0&0\end{pmatrix}$. Therefore, by Theorem \ref{pdi-thm}, $\omega+\ord(\Delta)\le J(\mathcal{E}_0)+e-\min\{\epsilon_{ij}\}=2+1-0=3$. On the other hand, by Proposition \ref{ord}, we have $\ord(\Delta)=er-\mu_{\omega}=3-2=1$. Therefore, $\omega+\ord(\Delta)=3$, which coincides with the above upper bound.
\end{example}

\end{document}